%% file: main.tex
\input{preamble}

\newcommand{\modulus}[1]{\left| #1 \right|}
\newcommand{\norm}[1]{\left\lVert #1 \right\rVert}
\newcommand{\integrate}[4]{\displaystyle\int_{#1}^{#2} #3\,#4}
\newcommand{\lip}[0]{\operatorname{Lip}}

\title{Coarse Embeddability Ratios of Banach Spaces}
\author{Avik Das}

\begin{document}

\maketitle

\input{00_abstract}

\tableofcontents

\input{01_introduction}

\input{02_prelim}
\input{03_property_q}
\input{04_tsirelson}
\input{05_metric_cotype}

\input{06_resolution}
\input{99_declaration}

\printbibliography

\end{document}

%% file: preamble.tex
\documentclass[reqno]{amsart}

\usepackage[utf8]{inputenc}
\usepackage[T1]{fontenc}
\usepackage{mathrsfs, comment, url, enumerate, float, marvosym, graphicx}
\usepackage[usenames,dvipsnames]{xcolor} 
\usepackage[normalem]{ulem}
\usepackage[all,arc,2cell]{xy}
\UseAllTwocells
\usepackage{tikz-cd}
\usepackage{amsmath}

\usepackage{hyperref}  
\hypersetup{
    colorlinks=true,
    linkcolor=blue,
    citecolor=blue,
    filecolor=blue,
    urlcolor=blue,
    pdfstartview=FitBH
}
\usepackage[noabbrev, nosort]{cleveref}

\usepackage[
    backend=biber,
    sorting=nyt,
    doi=false,
    url=false,
    isbn=false,
    eprint=false
]{biblatex}
\newtheorem{theorem}{Theorem}[section]
\newtheorem{proposition}[theorem]{Proposition}
\newtheorem{corollary}[theorem]{Corollary}
\newtheorem{lemma}[theorem]{Lemma}

\theoremstyle{definition}
\newtheorem{definition}[theorem]{Definition}
\newtheorem*{remark}{Note}
\newtheorem*{problem}{Problem}

\usepackage{geometry}


%% file: 00_abstract.tex
\begin{abstract}
    Given two Banach spaces $X$ and $E$, one can associate a numerical invariant $\mathcal{CR}(X, E)$, called the coarse embeddability ratio, which provides a criterion for coarse and uniform embeddability. We compute the coarse embeddability ratio for several important classes of Banach spaces, using various tools from the nonlinear theory of Banach spaces. Finally, we find pairs of Banach spaces with arbitrarily large coarse embeddability ratio, resolving an open problem of Rosendal in the negative.
\end{abstract}

%% file: 01_introduction.tex
\section{Introduction}
The traditional theory of Banach spaces was focused solely on their linear aspects. Indeed, the most natural maps to consider between two vector spaces are of course linear maps. However, one could also consider various other types of maps by ``forgetting'' various parts of the Banach space structure, such as the linear structure, the norm structure, or even the metric structure.

One of the first results concerning the nonlinear theory of Banach spaces is the Mazur-Ulam theorem of 1932, which states that any surjective isometry between real Banach spaces is an affine mapping~\cite{mazur1932transformationes}. Thus, linear structure of a Banach space is completely determined by its metric structure (up to choice of which point is distinguished to be $0$).

Several decades later, in 1967,  Kadec proved that all separable infinite-dimensional Banach spaces are homeomorphic ~\cite{kadets1967proof}. In 1982, Toru\'{n}czyk extended Kadec's result to show that any two Banach spaces with the same density character are homeomorphic ~\cite{torunczyk1981characterizing}. Thus, we see the linear structure of a Banach space is left completely undetermined by its topological structure.

These results suggest that somewhere in between keeping only topological structure and keeping the entire metric structure, we could hope for various types of maps which preserve some of the linear properties of Banach spaces. A big result in this direction is due to Ribe, which states that a uniform homeomorphism (see Section 2 for definitions) between Banach spaces preserve finite-dimensional subspaces ~\cite{ribe1978uniformly}.

Ribe's theorem suggests that ``local'' (that is, ``finite-dimensional'') properties of Banach spaces should be preserved under uniform homeomorphism, and thus, should be able to be formulated in purely metric terms. This led to the development of the Ribe program, in which local properties are formulated in metric terms, and then studied in the broader context of metric spaces and uniform embeddings.

In contrast to uniform embeddings, there is also a notion of coarse embedding, introduced by Gromov in \cite{gromov2007metric}. Coarse embeddings preserve ``large-scale structure'', and are thus in opposition to uniform embeddings. Despite this, the following problem has remained open.

\begin{problem}
    Let $X$ and $Y$ be Banach spaces. Are the following equivalent:
    \begin{enumerate}
        \item $X$ coarsely embeds into $Y$.
        \item $X$ uniformly embeds into $Y$.
    \end{enumerate}
\end{problem}

In fact, neither direction of the above problem is known. Several partial and related results have been obtained. In ~\cite{kalton2012uniform}, Kalton provided an example of two separable Banach spaces which are coarsely homeomorphic, but not uniformly homeomorphic. In \cite{randrianarivony2006characterization}, Randrianarivony proved that if $Y$ is a Hilbert space, then the above problem is resolved in the positive.

In ~\cite{rosendal2017equivariant}, Rosendal showed that if $X$ uniformly embeds into $Y$, then there is a simultaneously uniform and coarse embedding of $X$ into $\ell_p(Y)$ for all $1 \leq p < \infty$. Thus, if there is a coarse embedding from $\ell_p(Y)$ into $Y$, then a uniform embedding from $X$ into $Y$ implies a coarse embedding from $X$ into $Y$. In the same paper, Rosendal showed that if $X$ uniformly embeds into closed unit ball $B_Y$ of $Y$, then there is a uniformly continuous coarse embedding of $X$ into $Y \oplus Y$, providing another route to showing uniform embeddability implies coarse embeddability.

Following this, in ~\cite{rosendal2023separationratiosmapsbanach}, Rosendal introduced the notion of a coarse embeddability ratio between Banach spaces, and strengthened his previous result, showing that if $X$ uniformly embeds into $Y$ and $Y \oplus Y$ linearly embeds into $Y$, then $X$ coarsely embeds into $Y$.

On the other hand, far less progress has been made in showing that a coarse embedding implies a uniform embedding. In ~\cite{naor2015uniform}, Naor showed the existence of separable Banach spaces $X$ and $Y$, as well as a net $N$ of $X$ and a Lipschitz function $f:N \to Y$ such that for all uniformly continuous maps $F: X \to Y$,
\[
\sup\limits_{x \in N} \norm{F(x) - f(x)} = \infty.
\]
This suggests that proving the existence of a uniform embedding given a coarse embedding is perhaps less likely than the converse.

Here, we calculate several bounds on different coarse embeddability ratios between classical Banach spaces. In Section 2, we define separation ratios and coarse embeddability ratios and derive some lower bounds on coarse embeddability ratios. In Section 3, we recall Kalton's property $\mathcal{Q}$, and find an upper bound on the coarse embeddability ratio of a space without property $\mathcal{Q}$ into a space with property $\mathcal{Q}$. In Section 4, we establish upper bounds on the coarse embeddability ratios from important classes of spaces into Tsirelson's space, a classical (counter)example in the linear theory. In Section 5, we find an upper bound on the coarse embeddability ratio between $L_p(\mu)$ spaces. Finally, in Section 6, we resolve an open problem of Rosendal.

%% file: 02_prelim.tex
\section{Preliminaries on Separation Ratios}
We start by first defining certain moduli on maps between metric spaces. These moduli present alternative ways to view certain types of maps. Throughout this section, $(M_1, d_1)$ and $(M_2, d_2)$ are metric spaces.

\begin{definition}[Compression and Expansion Moduli]
    Let $\phi: M_1 \to M_2$ be an arbitrary (not necessarily continuous) map. The compression modulus $\kappa_\phi: (0, \infty) \to [0, \infty]$ and expansion modulus $\omega_\phi: (0, \infty) \to [0, \infty]$ of $\phi$ is given by
    \[
    \kappa_\phi(t) = \inf\limits_{d_1(x, y) \geq t} d_2(\phi(x), \phi(y))
    \]
    and
    \[
    \omega_\phi(t) = \sup\limits_{d_1(x, y) \leq t} d_2(\phi(x), \phi(y))
    \]
    for all $t \in (0, \infty)$. Additionally,
    \[
    \kappa(\phi) = \sup\{\kappa_\phi(t): t < \infty\} = \lim\limits_{t \to \infty} \kappa_\phi(t)
    \]
    and
    \[
    \omega(\phi) = \inf\{\omega_\phi(t): t > 0\} = \lim\limits_{t \to 0} \omega_\phi(t).
    \]
\end{definition}

These moduli provide quantitative ways to analyze several properties of maps. For example, a map $\phi:M_1 \to M_2$ is uniformly continuous if and only if $\omega(\phi) = 0$. Additionally for a map $\phi:X \to E$ between Banach spaces, $\omega(\phi) < \infty$ if and only if $\phi$ is coarsely Lipschitz, that is, there exists a constant $C$ such that for all $x, y \in X$,
\begin{equation*}
\label{eq:coarse_lip}
\norm{\phi(x) - \phi(y)} \leq C \norm{x - y} + C.
\end{equation*}
To see why, let $0 < R \leq 1$ such that $\omega_\phi(R) < \infty$, and we see that $C = \frac{\omega_\phi(R)}{R}$ will satisfy the above condition (take intermediate points separated by distance no more than $R$). In Banach spaces, it can be shown (see Lemma 14.1.15~\cite{albiac2016topics}) that $\phi$ is coarsely Lipschitz if and only if there are constants $C \geq 0$ and $R > 0$ such that for all $x, y \in X$ with $\norm{x - y} \geq R$,
\begin{equation*}
\label{eq:unnamed_map_type}
\tag{$\star$}
\norm{\phi(x) - \phi(y)} \leq C\norm{x - y}.
\end{equation*}
Thus, we see that a coarse Lipschitz map provides an upper bound on the expansion of large scale distances.

We remark this only holds for Banach spaces (or, more generally, unbounded metrically convex spaces). In the general metric space case, being coarsely Lipschitz always implies Condition (\ref*{eq:unnamed_map_type}), but the converse does not necessarily hold. Additionally, a map $\phi: M_1 \to M_2$ between metric spaces may satisfy $\omega(\phi) < \infty$ without being coarsely Lipschitz, since we can no longer take intermediate points.

A map $\phi:M_1 \to M_2$ between metric spaces is said to be a \textit{coarse} embedding if $\omega(\phi) < \infty$ and $\kappa(\phi) = \infty$. In this case, we see that $\phi$ is coarsely Lipschitz as before, but we additionally require that large scale distances remain large in the image.

On the other hand, a map $\phi:M_1 \to M_2$ between metric spaces is said to be a \textit{uniform} embedding if $\omega(\phi) = 0$ and $\kappa_\phi(t) > 0$ for all $t > 0$. Note that this requires a map to preserve small scale geometry, since $\omega(\phi) = 0$ enforces the uniform continuity of $\phi$, while $\kappa_\phi(t) > 0$ for all $t > 0$ ensures that distances are not shrunk entirely.

The following definitions provide a partial unification to these two dual notions.

\begin{definition}[Separation Ratio~\cite{rosendal2023separationratiosmapsbanach}]
    The separation ratio $\mathcal{R}(\phi)$ of a map $\phi: M_1 \to M_2$ between metric spaces is given by
    \[
    \mathcal{R}(\phi) = \frac{\kappa(\phi)}{\omega(\phi)},
    \]
    where $\frac{a}{\infty} = \frac{0}{a} = 0$ for all $a \in [0, \infty]$ and $\frac{a}{0} = \frac{\infty}{a} = \frac{\infty}{0} = \infty$ for all $a \in (0, \infty)$.
\end{definition}
\begin{definition}[Coarse Embeddability Ratio~\cite{rosendal2023separationratiosmapsbanach}]
    The coarse embeddability ratio of a metric space $M_1$ into a metric space $M_2$, denoted $\mathcal{CR}(M_1, M_2)$, is given by
    \[
    \mathcal{CR}(M_1, M_2) = \sup\limits_{\phi} \mathcal{R}(\phi),
    \]
    where the supremum is taken over all maps $\phi:M_1 \to M_2$.
\end{definition}

Observe that for a map $\phi:X \to E$ between Banach spaces, $\mathcal{R}(\phi) = \infty$ if and only if $\phi$ is a coarse embedding or uniformly continuous with $\kappa(\phi) > 0$.
In Theorem 3 of ~\cite{rosendal2023separationratiosmapsbanach}, it is shown that if $X$ and $E$ are Banach spaces such that $E \oplus E$ linearly embeds into $E$, then $\mathcal{CR}(X, E) = \infty$ if and only if there is a coarse embedding $\phi:X \to E$. In particular, we see that if $E \oplus E$ linearly embeds into $E$ and $X$ uniformly embeds into $E$, then $X$ coarsely embeds into $E$, showing that the coarse embeddability ratio provides some quantitative value of whether a space can be coarsely embedded into another. An open problem given by Rosendal in ~\cite{rosendal2023separationratiosmapsbanach} is the following.
\begin{problem}
    Let $X$ and $E$ be Banach spaces. Is it true that $\mathcal{CR}(X, E) > 1$ implies $\mathcal{CR}(X, E) = \infty$?
\end{problem}
In Section \ref{sec:resolution}, we resolve this problem in the negative. Thus, there are Banach spaces $X$ and $E$ for which  $X$ is ``partially'' coarsely embeddable into $E$.

The following lemma provides a more concrete view of the separation ratio. It provides the main tool for proving the nonexistence of maps with a certain separation ratio.
\begin{lemma}[Lemma 7~\cite{rosendal2023separationratiosmapsbanach}]
    \label{lemma2.4}
    Let $\phi:M_1 \to M_2$ be a map and let $K > 0$. Then
    \[
    \mathcal{R}(\phi) > K
    \]
    if and only if there exist constants $\Delta, \delta, \Lambda, \lambda > 0$ such that for all $x, y \in M_1$,
    \begin{equation}\label{eqn:lemma7_comp}
        d_1(x, y) \geq \Delta \Rightarrow d_2(\phi(x), \phi(y)) \geq \delta,\tag{$\ast$}
    \end{equation}
    \begin{equation}\label{eqn:lemma7_exp}
        d_1(x, y) \leq \Lambda \Rightarrow d_2(\phi(x), \phi(y)) \leq \lambda\tag{$\dag$}
    \end{equation}
    and $\frac{\delta}{\lambda} > K$.
\end{lemma}
\begin{proof}
    By definition, $\mathcal{R}(\phi) > K$ if and only if there exists $\Delta, \delta, \Lambda, \lambda > 0$ such that $\kappa_\phi(\Delta) \geq \delta$, $\omega_\phi(\Lambda) \leq \lambda$, and $\frac{\delta}{\lambda} > K$. Note $\kappa_\phi(\Delta) \geq \delta$ is equivalent to \cref*{eqn:lemma7_comp} and $\omega_\phi(\Lambda) \leq \lambda$ is equivalent to \cref*{eqn:lemma7_exp}, so the lemma holds.
\end{proof}
\begin{samepage}
\begin{remark}
    Throughout this paper, we will primarily focus on Banach spaces (or quasi-Banach spaces, such as $L_p(\mu)$ for $0 < p < 1$). Let $X$ and $E$ be Banach spaces and let $\phi:X \to E$ be a map. In the previous lemma, we may precompose and postcompose by dilations to obtain a map $\psi$ with $\mathcal{R}(\psi)$ in which we may choose specific values for one of $\Delta$ or $\Lambda$ and one of $\delta$ or $\lambda$. For example, let $\phi:X \to E$ be an arbitrary map and let $\Delta, \delta, \Lambda, \lambda > 0$ as in the previous lemma. Take $\Delta' = \frac{\Delta}{\Lambda}$, take $\delta' = \frac{\delta}{\lambda}$ and take $\psi: X \to E$ to be given by
    \[
    \psi(x) = \frac{\phi(\Lambda x)}{\lambda}.
    \]
    Then $\mathcal{R}(\psi) = \mathcal{R}(\phi)$ and for all $x, y \in X$, $\psi$ satisfies
    \begin{equation}
        \norm{x - y} \geq \Delta' \Rightarrow \norm{\psi(x) - \psi(y)} \geq \delta',\tag{$\ast'$}
    \end{equation}
    \begin{equation}
        \norm{x - y} \leq 1 \Rightarrow \norm{\psi(x) - \psi(y)} \leq 1\tag{$\dag'$}
    \end{equation}
    We use these types of arguments throughout, without any specific reference, and often just taking $\phi$ to be new translated map $\psi$ without explicit statement.
\end{remark}
\end{samepage}

We now establish some lower bounds on the coarse embeddability ratio between Banach spaces. The results rely on the existence of equilateral sets, which we now define.
\begin{definition}[Equilateral Sets]
    Let $X$ be a Banach space and let $\lambda > 0$. A subset $A \subseteq X$ is said to be $\lambda$-equilateral if for all $x, y \in A$ with $x \neq y$, we have $\norm{x - y} = \lambda$.
\end{definition}
\begin{proposition}
    \label{prop2.6}
    Let $X$ be a Banach space with density character $\tau$ (the smallest cardinality of a dense subset in $X$) and let $E$ be a Banach space with a $1$-equilateral set $A$ of cardinality $\tau$. Then $\mathcal{CR}(X, E) \geq 1$.
\end{proposition}
\begin{proof}
    Since $X$ has density character $\tau$, we may partition $X$ into subsets $\{X_a\}_{a \in A}$ of diameter at most $1$. For each $a \in A$ and each $x \in X_a$, define $\phi(x) = a$, obtaining a map $\phi:X \to E$. If $\norm{x_1 - x_2} > 1$ for some $x_1, x_2 \in X$, then $x_1$ and $x_2$ are contained in distinct parts of the partition $\{X_a\}_{a \in A}$, so $\norm{\phi(x_1) - \phi(x_2)} = 1$. Thus, $\kappa(\phi)  = 1$. Additionally, for all $x_1, x_2 \in X$, we see $\norm{\phi(x_1) - \phi(x_2)} \leq 1$, so $\omega(\phi) \leq 1$. Thus, $\mathcal{R}(\phi) \geq 1$, so $\mathcal{CR}(X, E) \geq 1$.
\end{proof}
\begin{samepage}
\begin{proposition}
    \label{prop2.7}
    Let $X$ be a separable Banach space and let $E$ be an infinite-dimensional Banach space. Then $\mathcal{CR}(X, E) \geq 1$.
\end{proposition}
\begin{proof}
    Let $\epsilon > 0$. By Theorem 1 of ~\cite{mercourakis2013equilateralsetsinfinitedimensional}, there exists a Banach space $Y$ with a linear isomorphism $S: Y \to E$ satisfying $\norm{S}\norm{S^{-1}} < 1 + \epsilon$ such that $Y$ contains an infinite $1$-equilateral set . By Proposition \ref{prop2.6}, there exists a map $\phi:X \to Y$ with $\mathcal{R}(\phi) \geq 1 - \epsilon$.
    Set $T = \norm{S^{-1}} \cdot S$. Clearly, 
    \[
    \norm{T^{-1}} = \norm{S^{-1}\norm{S^{-1}}^{-1}} = \norm{S^{-1}}\norm{S^{-1}}^{-1} = 1.
    \]
    Additionally,
    \[
    \norm{T} = \norm{S}\norm{S^{-1}} < 1 + \epsilon.
    \]
    Now, for all $x, y \in X$,
    \[
    \norm{\phi(x) - \phi(y)} \leq \norm{T(\phi(x)) - T(\phi(y))} \leq (1 + \epsilon)\norm{\phi(x) - \phi(y)}.
    \]
    Thus,
    \[
    \mathcal{R}(T \circ \phi) \geq \frac{\kappa(\phi)}{(1 + \epsilon) \omega(\phi)} \geq \frac{1}{1 + \epsilon}\mathcal{R}(\phi) \geq \frac{1 - \epsilon}{1 + \epsilon}.
    \]
    Since $\epsilon > 0$ is arbitrary, we see that $\mathcal{CR}(X, E) \geq 1$.
\end{proof}
\end{samepage}
The final result of this section establishes the coarse embeddability ratio into a finite-dimensional space. This completes several cases in later results (particularly, in Section 5, where we ignore finite-dimensional $L_p(\mu)$ spaces).
\begin{proposition}
    \label{prop2.8}
    Let $X$ and $E$ be Banach spaces with $E$ finite-dimensional and $\dim(X) > \dim(E)$. Then $\mathcal{CR}(X, E) = 0$.
\end{proposition}
\begin{proof}
    Let $m = \dim(E)$. Let $Z \subseteq X$ be a finite-dimensional subspace of $X$ with $\dim(Z) > m$ and set $n = \dim(Z)$. Since $Z$ is contained within $X$, it suffices to show that $\mathcal{CR}(Z, E) = 0$ (since we may just restrict every map $X \to E$ to a map $Z \to E$). WLOG, we may suppose $Z = \mathbb{R}^n$ and $E = \mathbb{R}^m$, as every norm on a finite-dimensional vector space is equivalent. Let $\{e_1, ..., e_n\}$ denote the standard basis of $\mathbb{R}^n$. Additionally, throughout this proof, for each $k \in \mathbb{N}$, $\mu_k$ denotes the $k$-dimensional Lebesgue measure and $C_k$ denotes the constant so $\mu_k(B_{\mathbb{R}^k}) = C_k$ (where $B_{\mathbb{R}^k}$ is the unit ball in $\mathbb{R}^k$).\\\\
    Suppose $\phi: \mathbb{R}^n \to \mathbb{R}^m$ with $\mathcal{R}(\phi) > 0$. By Lemma \ref{lemma2.4}, there exist $\Delta, \delta, \Lambda, \lambda > 0$ such that $\frac{\delta}{\lambda} > 0$ and for all $x, y \in \mathbb{R}^n$,
    \[
    \norm{x - y} \geq \Delta \Rightarrow \norm{\phi(x) - \phi(y)} \geq \delta
    \]
    \[
    \norm{x - y} \leq \Lambda \Rightarrow \norm{\phi(x) - \phi(y)} \leq \lambda.
    \]
    Let $N \in \mathbb{N}$ such that $\frac{\Delta}{N} \leq \Lambda$. For each $k \in \mathbb{N}$, set
    \[
    X_k = \left\{\sum\limits_{i = 1}^n \frac{\Delta a_i e_i}{N}: (a_1, ..., a_n) \in \{0, 1, ..., N(k - 1)\}^n\right\}
    \]
    and let
    \[
    Y_k = \left\{\sum\limits_{i = 1}^n \Delta a_i e_i: (a_1, ..., a_n) \in \{0, 1, ..., (k - 1)\}^n\right\}.
    \]
    Let $k \in \mathbb{N}$. Note $Y_k \subseteq X_k$. For each $x \in Y_k$, let $V_x$ be the open ball in $\mathbb{R}^m$ centered at $\phi(x)$ of radius $\delta/2$. For each $x, y \in Y_k$, we see $\norm{x - y} \geq \Delta$, so $\norm{\phi(x) - \phi(y)} \geq \delta$. Thus, the $V_x$ are disjoint.
    Let $x_0, y_0 \in X_k$ satisfy 
    \[
    R_k := \norm{\phi(x_0) - \phi(y_0)} = \max\limits_{z_1, z_2 \in X_k} \norm{\phi(z_1) - \phi(z_2)}.
    \]
    Then, if $B_{\mathbb{R}^m}(\phi(x_0), R_k + \delta/2))$ denotes the closed ball centered at $\phi(x_0)$ with radius $R_k + \delta/2$ in $\mathbb{R}^m$, we see
    \[
    \bigcup\limits_{x \in Y_k} V_x \subseteq B_{\mathbb{R}^m}(\phi(x_0), R_k + \delta/2),
    \]
    so since the $V_x$ are disjoint,
    \[
    C_m(\delta/2)^m k^n = \sum\limits_{x \in Y_k} \mu_m(V_x) \leq \mu_m(B_{\mathbb{R}^m}(\phi(x_0), R_k + \delta/2)) = C_m (R_k + \delta/2)^m
    \]
    Thus, $R_k \geq \frac{\delta}{2} (k^{n/m} - 1)$. Now, observe that for $x, y \in X_k$, we see
    \[
    \norm{x - y} \leq \left(\sum\limits_{i = 1}^n (N(k - 1)\Lambda)^2\right)^{1/2} = N(k - 1)\Lambda\sqrt{n}.
    \]
    Set $M_k = \lceil N(k - 1) \sqrt{n} \rceil$. Then we may put points $z_0, z_1 ..., z_{M_k}$ with $x = z_0$ and $y = z_{M_k}$ such that $\norm{z_j - z_{j - 1}} \leq \Lambda$ for all $1 \leq j \leq M_k$. Thus,
    \[
    \norm{\phi(x) - \phi(y)} \leq \sum\limits_{j = 1}^{M_k}\norm{\phi(z_j) - \phi(z_{j - 1})} \leq M_k\lambda.
    \]
    Thus, $R_k \leq M_k \lambda$. Since $k$ is arbitrary and $n/m > 1$, we make take $k \in \mathbb{N}$ sufficiently large such that
    \[
    Nkn \lambda < \frac{\delta}{2}(k^{n/m} - 1)
    \]
    Note $M_k \leq Nkn$. Then we see
    \[
    R_k \leq M_k \lambda \leq N kn \lambda < \frac{\delta}{2}(k^{n/m} - 1) \leq R_k,
    \]
    a contradiction. Thus, $\mathcal{R}(\phi) = 0$. Thus, $\mathcal{CR}(\mathbb{R}^n, \mathbb{R}^m) = 0$.
\end{proof}
Of course, if $n \leq m$, then $\mathbb{R}^n$ simultaneously coarsely and uniformly embeds into $\mathbb{R}^m$ via the inclusion map. Thus, in combination with the above proposition, if $E$ is finite-dimensional, we have 
\[
\mathcal{CR}(X, E) = \begin{cases}
    \infty & \dim(X) \leq \dim(E)\\
    0 & \dim(X) > \dim(E).
\end{cases}
\]

%% file: 03_property_q.tex
\section{Coarse Embeddability Ratios and Kalton's Property \texorpdfstring{$\mathcal{Q}$}{Q}}
In \cite{kalton2007coarse}, Kalton proved that $c_0$ does not coarsely or uniformly embed into any reflexive Banach space. During his investigations into this problem, he formulated the $\mathcal{Q}$-property, which concerns the embeddability of certain graphs on tuples of natural numbers into a Banach space. We define this property in the following definitions.
\begin{definition}
    Let $\mathbb{M} \subseteq \mathbb{N}$. Then for all $r \in \mathbb{N}$, $[\mathbb{M}]^r$ denotes the collection of all subsets of $\mathbb{M}$ with cardinality $r$. Additionally, $[\mathbb{M}]^\omega$ denotes the collection of all infinite subsets of $\mathbb{M}$.
\end{definition}
\begin{definition}[The Kalton Interlacing Graph~\cite{kalton2007coarse}]
    The Kalton Interlacing graph $K_r(\mathbb{N})$ makes $[\mathbb{N}]^r$ into a graph by saying that two distinct elements  $A = \{a_1, ..., a_r\}$ and $B = \{b_1, ..., b_r\}$ (both written in increasing order) are connected by an edge if they interlace:
    \[
    a_1 \leq b_1 \leq a_2 \leq \cdots \leq a_r \leq b_r  \quad \text{ or }\quad b_1 \leq a_1 \leq b_2 \leq \cdots \leq b_r \leq a_r.
    \]
    Note this graph is connected (this can be seen by swapping the minimal disagreeing elements at most $r$ times). This defines a natural shortest path metric on $K_r(\mathbb{N})$, which we denote by $d_K$.
\end{definition}
\begin{definition}[Property $\mathcal{Q}(\epsilon, \delta)$~\cite{kalton2007coarse}]
    Let $\epsilon, \delta > 0$. A metric space $(M, d)$ has property $\mathcal{Q}(\epsilon, \delta)$ if for all $r \in \mathbb{N}$ and all $f:K_r(\mathbb{N}) \to M$ with $\omega_f(1) \leq \delta$, there exists $\mathbb{M} \in [\mathbb{N}]^\omega$ such that for all $A, B \in [\mathbb{M}]^r$ with $\max(A) < \min(B)$, we have
    \[
    d(f(A), f(B)) \leq \epsilon.
    \]
    Additionally, for all $\epsilon > 0$, we set
    \[
    \Delta_M(\epsilon) = \sup\{\delta \geq 0: \text{$M$ has property $\mathcal{Q}(\epsilon, \delta)$}\}.
    \]
\end{definition}
\begin{definition}[Property $\mathcal{Q}$~\cite{kalton2007coarse}]
    Let $(M, d)$ be a metric space. The property $\mathcal{Q}$ constant of $M$, denoted $\mathcal{Q}_M$, is given by
    \[
    \mathcal{Q}_M = \sup\{c \geq 0: \Delta_M(\epsilon) \geq c\epsilon \text{ for all $\epsilon > 0$}\}.
    \]
    We say $M$ has property $\mathcal{Q}$ if $Q_M > 0$.
\end{definition}
The following lemma establishes the scaling of $\Delta_X(\epsilon)$ for a Banach space $X$. It relies on a simple scaling argument.
\begin{lemma}
\label{lemma3.5}
    Let $X$ be a Banach space. Then, for all $\epsilon > 0$, $\Delta_X(\epsilon) = \mathcal{Q}_X(\epsilon)$.
\end{lemma}
\begin{proof}
    Let $\epsilon > 0$. By definition of $\mathcal{Q}_X$, we see $\Delta_X(\epsilon) \geq \mathcal{Q}_X\epsilon$. Now, suppose there exists $\alpha > \mathcal{Q}_X$ such that $X$ has property $\mathcal{Q}(\epsilon, \alpha \epsilon)$. Let $\eta > 0$ and let $f:K_r(\mathbb{N}) \to X$ be a map with $\omega_f(1) \leq \alpha \eta$. Consider the map $g = \frac{\epsilon}{\eta}f$. Then $\omega_g(1) \leq \alpha \epsilon$, so since $X$ has property $\mathcal{Q}(\epsilon, \alpha \epsilon)$, there exists $\mathbb{M} \in [\mathbb{N}]^\omega$ such that for all $A, B \in [\mathbb{M}]^r$ with $\max(A) < \min(B)$, we have 
    \[
    d(g(A), g(B)) \leq \epsilon.
    \]
    Thus, for all $A, B \in [\mathbb{M}]^r$ with $\max(A) < \min(B)$, we see 
    \[
    d(f(A), f(B)) \leq \eta,
    \]
    so $X$ has property $\mathcal{Q}(\eta, \alpha \eta)$. Thus, for all $\eta > 0$, we have $\Delta_X(\eta) \geq \alpha \eta$. This means $\alpha \leq \mathcal{Q}_X$, a contradiction. It follows that $\Delta_X(\epsilon) \leq \mathcal{Q}_X \epsilon$, so $\Delta_X(\epsilon) = \mathcal{Q}_X\epsilon$.
\end{proof}
Part of the significance of Kalton's work is that it allows for the use Ramsey theory and related combinatorial methods to solve coarse and uniform embedding problems.
It also inspired investigating embeddings of other traditional graphs (the Hamming graph, the Johnson graph, etc.), which will be relied upon in the next section. The following is an important example of spaces with property $\mathcal{Q}$.
\begin{definition}[Stable Metric Space]
    A metric space $(M, d)$ is said to be stable if for all sequences $\{x_n\}_{n = 1}^\infty, \{y_m\}_{m = 1}^\infty$ in $M$,
    \[
    \lim\limits_{n \to \infty} \lim\limits_{m \to \infty} d(x_n, y_m) = \lim\limits_{m \to \infty} \lim\limits_{n \to \infty} d(x_n, y_m),
    \]
    provided both limits exist.
\end{definition}

In \cite{kalton2007coarse}, Kalton shows that reflexive Banach spaces and stable metric spaces both satisfy property $\mathcal{Q}$. Additionally, he showed that if a Banach space $X$ coarsely or uniformly embeds into a space satisfying property $\mathcal{Q}$, then $X$ itself must have property $\mathcal{Q}$.

In the following proposition, we use Kalton's property $\mathcal{Q}$ to establish an upper bound on the separation ratio between a space not having property $\mathcal{Q}$ and a space having property $\mathcal{Q}$.

\begin{proposition}
    \label{prop3.5}
    Let $X$ be a Banach space not satisfying the $\mathcal{Q}$-property and let $E$ be a Banach space satisfying the $\mathcal{Q}$-property. Then $\mathcal{CR}(X, E) \leq 1/\mathcal{Q}_E$. 
\end{proposition}
\begin{proof}
    First, note since $E$ has the $\mathcal{Q}$-property, we see $E$ satisfies the $\mathcal{Q}(1, \mathcal{Q}_E - \epsilon)$-property for all $0 < \epsilon < \mathcal{Q}_E$. Let $0 < \epsilon < \mathcal{Q}_E$ and suppose $\phi: X \to E$ with $\mathcal{R}(\phi) > 1/(\mathcal{Q}_E - \epsilon)$. Using Lemma \ref{lemma2.4}, let $\Delta > 0$ and $\delta > 1$ such that for all $x, y \in X$,
    \[
    \norm{x - y} \geq \Delta \Rightarrow \norm{\phi(x) - \phi(y)} \geq \delta
    \]
    \[
    \norm{x - y} \leq 1 \Rightarrow \norm{\phi(x) - \phi(y)} \leq \mathcal{Q}_E - \epsilon.
    \]
    Note this ensures $\omega_\phi(1) \leq \mathcal{Q}_E - \epsilon$.\\\\
    Since $X$ does not satisfy the $\mathcal{Q}$-property, we see that $X$ does not have the $\mathcal{Q}(\Delta, 1)$ property by Lemma \ref{lemma3.5}. Thus, there exists a map $f:K_r(\mathbb{N}) \to X$ with $\omega_f(1) \leq 1$ such that for every infinite subset $\mathbb{M}$ of $\mathbb{N}$, there exists $C, D \in [\mathbb{M}]^r$ with $\max(C) < \min(D)$ such that $\norm{f(C) - f(D)} > \Delta$.\\\\
    Since $\omega_{\phi}(1) \leq \mathcal{Q}_E - \epsilon$ and $\omega_f(1) \leq 1$, we see that $\omega_{\phi \circ f}(1) \leq \mathcal{Q}_E - \epsilon$. Thus, there exists an infinite subset $\mathbb{M}$ of $\mathbb{N}$ such that for all $A, B \in [\mathbb{M}]^r$ with $\max(A) < \min(B)$,
    \[
    \norm{\phi(f(A)) - \phi(f(B))} \leq 1.
    \]
    Let $C, D \in [\mathbb{M}]^r$ with $\max(C) < \min(D)$ and $\norm{f(C) - f(D)} > \Delta$. Then
    \[
    \norm{\phi(f(C)) - \phi(f(D))} \geq \delta > 1,
    \]
    a contradiction.
    Thus, $\mathcal{CR}(X, E) \leq 1/(\mathcal{Q}_E - \epsilon)$ for all $\epsilon > 0$. Since $\epsilon > 0$ is arbitrary, we see $\mathcal{CR}(X, E) \leq 1/\mathcal{Q}_E$.
\end{proof}
\begin{corollary}
    Let $E$ be a reflexive Banach space. Then $\mathcal{CR}(c_0, E) \leq 2$.
\end{corollary}
\begin{proof}
    Theorem 4.1 of \cite{kalton2007coarse} shows $\mathcal{Q}_E \geq 1/2$ and Theorem 3.5 of \cite{kalton2007coarse} shows that $c_0$ does not have the $\mathcal{Q}$-property. Thus, by Proposition \ref{prop3.5}, 
    $\mathcal{CR}(c_0, E) \leq 1/\mathcal{Q}_E \leq 2$.
\end{proof}
\begin{corollary}
    Let $E$ be a stable Banach space. Then $\mathcal{CR}(c_0, E) \leq 1$. In particular, if $E$ is infinite-dimensional, then $\mathcal{CR}(c_0, E) = 1$.
\end{corollary}
\begin{proof}
    By Proposition 3.1 of \cite{kalton2007coarse}, we see that $\mathcal{Q}_E = 1$ and Theorem 3.5 of \cite{kalton2007coarse} shows $c_0$ does not have the $\mathcal{Q}$-property. Thus, by Proposition \ref{prop3.5}, $\mathcal{CR}(c_0, E) \leq 1$. In the particular case, if $E$ is infinite-dimensional, then since $c_0$ is separable, we see by Proposition \ref{prop2.7}, we see $\mathcal{CR}(c_0, E) = 1$.
\end{proof}

%% file: 04_tsirelson.tex
\section{Coarse Embeddability Ratios and Tsirelson's Space}
Our next example concerns the coarse embeddability ratio between $\ell_p$ spaces and Tsirelson's space $\mathcal{T}^*$.
Tsirelson's space was originally constructed by Tsirelson in \cite{tsirel1974not}, providing the first example of a Banach space containing no (linearly) isomorphic copies of $c_0$ and $\ell_p$, $1 \leq p \leq \infty$. Shortly after, Figiel and Johnson published a related paper explaining several properties of the dual to Tsirelson's space~\cite{figiel1974uniformly}.
Today, $\mathcal{T}$ is used for the space in Figiel and Johnson's construction, while $\mathcal{T}^*$ is used for the space originally constructed by Tsirelson. This notation is justified due to the fact Tsirelson's space is reflexive. Below, we present the standard recursive norm definition for $\mathcal{T}$ given by Figiel and Johnson (although we make no attempt to verify any of the earlier properties or its well-definedness).
\begin{definition}[Tsirelson's Space]
    A collection $\{I_1, ..., I_m\}$ of disjoint intervals of natural numbers is said to be admissible if $m < \min(I_k)$ for $1 \leq k \leq m$. We define an inductive family of norms on $c_{00}$. For $x \in c_{00}$, we define $\norm{x}_0 = \norm{x}_{c_0}$. For $n \in \mathbb{N}$, we define
    \[
    \norm{x}_n = \max\left\{\norm{x}_{n - 1}, \sup \frac{1}{2}\sum\limits_{j = 1}^m \norm{P_{I_j} x}_{n - 1}\right\},
    \]
    where the supremum is taken over all admissible families of intervals and $P_{I_j} x$ represents the projection of $x$ onto the coordinates present in the interval $I_j$. Finally, we define
    \[
    \norm{x}_\mathcal{T} = \sup\limits\{\norm{x}_n: n \geq 0\},
    \]
    and let $\mathcal{T}$ be the completion of $c_{00}$ under the norm $\norm{\,\cdot\,}_\mathcal{T}$.
\end{definition}

In \cite{Baudier_2018}, Baudier et al. provided a strengthening of Tsirelson's original result: not only do $c_0$ and $\ell_p$, $1 \leq p \leq \infty$ not linearly embed into $\mathcal{T}^*$, they also do not \textit{coarsely} embed into $\mathcal{T}^*$. 
Their arguments rely on analyzing the embeddability of two types of graphs on tuples of natural numbers, which we define here.
\begin{definition}[The Johnson Graph~\cite{Baudier_2018}]
    The Johnson graph $J_r(\mathbb{N})$ makes $[\mathbb{N}]^r$ into a graph by saying that two distinct elements $A, B \in [\mathbb{N}]^r$ are connected if 
    \[
    \modulus{A \bigtriangleup B} = 2.
    \]
    As with the Kalton interlacing graph, this graph is connected. Note the shortest path metric $d_J$ on $J_r(\mathbb{N})$ is given by
    \[
    d_J(A, B) = \frac{\modulus{A \bigtriangleup B}}{2}.
    \]
\end{definition}
\begin{definition}[The Hamming Graph~\cite{Baudier_2018}]
    The Hamming graph $H_r(\mathbb{N})$ makes $[\mathbb{N}]^r$ into a graph by saying that two distinct elements $A = \{a_1, ..., a_r\}$ and $B = \{b_1, ..., b_r\}$ (both written in increasing order) are connected by an edge if
    \[
    \modulus{\{1 \leq i \leq r: a_i \neq b_i\}} = 1.
    \]
    Similar to the Johnson and Kalton graphs, this graph is connected. The shortest path metric $d_H$ on $H_r(\mathbb{N})$ is given by
    \[
    d(\{a_1, ..., a_r\}, \{b_1, ..., b_r\}) = \modulus{\{1 \leq i \leq r: a_i \neq b_i\}}.
    \]
\end{definition}
As mentioned earlier, the use of such graphs is that we may apply Ramsey-style theorems to them. We now state two such Ramsey style results, both proven in \cite{Baudier_2018}. In order to do so, it is important to first establish the concept of bimonotone bases and the support of a vector.
\begin{definition}[Schauder Basis]
    Let $X$ be a Banach space. A sequence $\{x_n\}_{n = 1}^\infty$ is said to be a Schauder basis for $X$ if for all $x \in X$, there is a unique sequence of scalars $\{\alpha_n\}_{n = 1}^\infty$ such that
    \[
    x = \sum\limits_{n = 1}^\infty \alpha_n x_n.
    \]
\end{definition}
\begin{definition}[Bimonotone Basis]
    Let $X$ be a Banach space with Schauder basis $\{x_n\}_{n = 1}^\infty$. For each $k \in \mathbb{N}$, the $k$th natural projection $P_k$ associated with $\{x_n\}_{n = 1}^\infty$ is the linear map $P_k: X \to X$ given by
    \[
    P_k\left(\sum\limits_{n = 1}^\infty \alpha_n x_n\right) = \sum\limits_{n = 1}^k \alpha_n x_n.
    \]
    We say $\{x_n\}_{n = 1}^\infty$ is a bimonotone basis if for all $k \in \mathbb{N}$,
    \[
    \norm{P_k} = \norm{I - P_k} = 1,
    \]
    where $I:X \to X$ denotes the identity operator.
\end{definition}
\begin{definition}[Support of a Vector]
    Let $X$ be a Banach space with Schauder basis $\{x_i\}_{i = 1}^\infty$. The support of a vector $x = \sum_{i = 1}^n \alpha_i x_i$ is given by
    \[
    \operatorname{supp}(x) = \{i \in \mathbb{N}: \alpha_i \neq 0\}.
    \]
    For $x, y \in X$ and $r \in \mathbb{N}$, we write $x \prec y$ to mean $\max(\operatorname{supp}(x)) < \min(\operatorname{supp}(y))$ and $r \preceq x$ to mean $r \leq \min(\operatorname{supp}(x))$, where all supports are with respect to the basis being considered.
\end{definition}
\begin{theorem}[Proposition 4.1, \cite{Baudier_2018}]
    \label{thm4.3}
    Let $Y$ be a reflexive Banach space with bimonotone basis $\{e_i\}_{i = 1}^\infty$. For all $k, r \in \mathbb{N}$, $\epsilon > 0$, and $\mathbb{M} \in [\mathbb{N}]^\omega$, and Lipschitz maps $f:([\mathbb{M}]^k, d_\bullet) \to Y$, where $d_\bullet$ can be either $d_H$ or $d_J$, there exists $\mathbb{M}' \in [\mathbb{N}]^\omega$ and $y \in Y$ satisfying the following:
    For all $A \in [\mathbb{M}']^k$, there exists $r \preceq y_A^{(1)} \prec y_A^{(2)} \prec \cdots \prec y_A^{(k)}$, each with finite support with respect to $\{e_i\}_{i = 1}^\infty$ such that
    \[
    \norm{y_A^{(i)}} \leq \omega_f(1)
    \]
    and
    \[
    \norm{f(A) - (y + y_A^{(1)} + y_A^{(2)} + \cdots + y_A^{(k)})} < \epsilon.
    \]
\end{theorem}
We remark that $\mathcal{T}^*$ is a reflexive Banach space with bimonotone basis, so we may apply Theorem \ref{thm4.3} with $Y = \mathcal{T}^*$.
\begin{theorem}[Theorem 4.4, \cite{Baudier_2018}]
    \label{thm4.4}
    Let $r \in \mathbb{N}$ and let $f: ([\mathbb{N}]^r, d_\bullet) \to \mathcal{T}^*$ be Lipschitz, where $\mathcal{T}^*$ denotes Tsirelson's space. Then there exists $\mathbb{M}' \in [\mathbb{N}]^\omega$ such that for all $A, B \in [\mathbb{M}']^r$,
    \[
    \norm{f(A) - f(B)} \leq 5\omega_f(1).
    \]
\end{theorem}
Using these results, we can now establish two corresponding upper bounds on coarse embeddability ratios into Tsirelson's space.
\begin{samepage}
\begin{proposition}
    Let $\mathcal{T}^*$ be Tsirelson's space and let $0 < p < \infty$. Then $\mathcal{CR}(\ell_p, \mathcal{T}^*) \leq 4$.
\end{proposition}
\begin{proof}
    Suppose $\phi: \ell_p \to \mathcal{T}^*$ with $R(\phi) > 4$. By Lemma \ref{lemma2.4}, we may obtain $\Delta > 0$ and $\delta > 4$ such that for all $x, y \in \ell_p$,
    \[
    \norm{x - y} \geq \Delta \Rightarrow \norm{\phi(x) - \phi(y)} \geq \delta
    \]
    and
    \[
    \norm{x - y} \leq 2^{1/p} \Rightarrow \norm{\phi(x) - \phi(y)} \leq 1.
    \]
    Let $\{e_n\}_{n = 1}^\infty$ be the sequence of standard unit vectors in $\ell_p$ and set $\epsilon = \frac{\delta - 4}{2}$. Let $r \in \mathbb{N}$ such that $(2r)^{1/p} \geq \Delta$. Define $f:J_r(\mathbb{N}) \to \mathcal{T}^*$ by
    \[
    f(A) = \phi\left(\sum\limits_{n \in A} e_n\right)
    \]
    for all $A \in [\mathbb{N}]^r$.\\\\
    Suppose $d_J(A, B) = 1$. Then $\norm{\sum\limits_{n \in A} e_n - \sum\limits_{n \in B} e_n} = 2^{1/p}$, so $\norm{f(A) - f(B)} \leq 1$. Thus, $f$ is Lipschitz.\\\\
    Applying Theorem \ref{thm4.3} with $\mathbb{M} = \mathbb{N}$, we may obtain $\mathbb{M}' \in [\mathbb{N}]^\omega$ and $y \in \mathcal{T}^*$. Take $A, B \in [\mathbb{M}']^r$. Then
    \begin{equation*}
        \begin{aligned}
            \norm{f(A) - f(B)} &< \norm{y_A^{(1)} + \cdots + y_A^{(r)}} + \norm{y_B^{(1)} + \cdots + y_B^{(r)}} + 2\epsilon\\
            &\leq 4 + 2\epsilon\\
            & = \delta,
        \end{aligned}
    \end{equation*}
    where the last inequality is justified by the fact $\norm{y_C^{(1)} + ... + y_C^{(r)}} \leq 2$ for all $C \in [\mathbb{M}']^r$, a consequence of (2.13) in \cite{Baudier_2018}.
    Now, suppose $A, B \in [\mathbb{M}']^r$ are disjoint. Then
    \[
    \norm{\sum\limits_{n \in A} e_n - \sum\limits_{n \in B} e_n} = (2r)^{1/p} \geq \Delta,
    \]
    so
    \[
    \norm{f(A) - f(B)} \geq \delta,
    \]
    a contradiction. Thus, $\mathcal{R}(\phi) \leq 4$, so $\mathcal{CR}(\ell_p, \mathcal{T}^*) \leq 4$.
\end{proof}
\end{samepage}
\begin{proposition}
    Let $X$ be a nonreflexive Banach space and let $\mathcal{T}^*$ denote Tsirelson's space. Then $\mathcal{CR}(X, \mathcal{T}^*) \leq 5$. In particular, we may take $X = c_0$ or $X = \ell_\infty$.
\end{proposition}
\begin{proof}
    Suppose $\phi:X \to \mathcal{T}^*$ with $\mathcal{R}(\phi) > 5$. By precomposing with dilations, we may apply Lemma \ref{lemma2.4} to obtain $\Delta > 0$ and $\delta > 5$ such that for all $x, y \in X$,
    \[
    \norm{x - y} \geq \Delta \Rightarrow \norm{\phi(x) - \phi(y)} \geq \delta
    \]
    and
    \[
    \norm{x - y} \leq 2 \Rightarrow \norm{\phi(x) - \phi(y)} \leq 1.
    \]
    By a characterization of reflexive spaces by James~\cite{James1972SOMESP}, there exists a sequence $\{x_i\}_{i = 1}^\infty$ in $B_X$ such that for all $r \in \mathbb{N}$ and $A = \{a_1, ..., a_{2r}\} \in [\mathbb{N}]^{2r}$ (written in increasing order),
    \[
    \norm{\sum\limits_{i = 1}^r x_{a_i} - \sum\limits_{i = r + 1}^{2r} x_{a_i}} \geq \frac{r}{2}.
    \]
    Let $r \geq 2\Delta$. Define $s:H_r(\mathbb{N}) \to X$ by $s(A) = \sum\limits_{n \in A} x_n$ for all $A \in H_r(\mathbb{N})$ and define $f:H_r(\mathbb{N}) \to \mathcal{T}^*$ by
    \[
    f(A) = (\phi \circ s)(A) = \phi\left(\sum\limits_{n \in A} x_n\right)
    \]
    for all $A \in H_r(\mathbb{N})$. It is clear that the map $s$ is $2$-Lipschitz and $f$ is $\omega_{f}(1)$-Lipschitz. Since
    \[
    \omega_{f}(1) = \omega_{\phi \circ s}(1) \leq \omega_{\phi}(2),
    \]
    we see $\lip(f) \leq \omega_{\phi}(2)$. Applying Theorem \ref{thm4.4}, we obtain some $\mathbb{M}' \in [\mathbb{N}]^\omega$ such that for all $A, B \in [\mathbb{M}']^r$,
    \[
    \norm{f(A) - f(B)} \leq 5\omega_f(1) \leq 5\omega_{\phi}(2) \leq 5 < \delta.
    \]
    Now, let $A, B \in [\mathbb{M}']^r$ with $\max(A) < \min(B)$. Applying James' characterization, we see
    \[
    \norm{\sum\limits_{n \in A} x_n - \sum\limits_{n \in B} x_n} \geq \frac{r}{2} \geq \Delta.
    \]
    Thus,
    \[
    \norm{f(A) - f(B)} \geq \norm{\phi\left(\sum\limits_{n \in A} x_n\right) - \phi\left(\sum\limits_{n \in B} x_n\right)} \geq \delta,
    \]
    a contradiction. Thus, $\mathcal{CR}(X, \mathcal{T}^*) \leq 5$.
\end{proof}

%% file: 05_metric_cotype.tex
\section{Coarse Embeddability Ratios and Metric Cotype}
Our last example concerns embeddability ratios between $L_p(\mu)$ spaces. Our results rely on the machinery of metric cotype developed by Mendel and Naor in \cite{mendel2008metric}. Throughout this section, when we speak of a $L_p(X, \Sigma, \mu)$ space (generally abbreviated $L_p(\mu)$ or simply $L_p$), we assume $(X, \Sigma, \mu)$ is not a purely atomic measure space consisting of a finite number of atoms. This ensures our $L_p$ spaces are always infinite-dimensional. As mentioned earlier, the finite-dimensional cases are handled by Proposition \ref{prop2.8} and the remark following it.

In the linear theory of Banach spaces, two important invariants preserved by linear embedding are Rademacher type and cotype, which we define now.
\begin{definition}[Rademacher type]
    Let $X$ be a Banach space and let $p > 0$. The space $X$ is said to have Rademacher type $p$ if there exists a constant $T < \infty$ such that for all $n \in \mathbb{N}$ and $x_1, ..., x_n$,
    \[
    \mathbb{E}_\sigma\norm{\sum\limits_{j = 1}^n \sigma_j x_j}^p \leq T^p \sum\limits_{j = 1}^n \norm{x_j}^p,
    \]
    where expectation is taken over uniformly chosen $\sigma \in \{-1, 1\}^n$.
\end{definition}
\begin{definition}[Rademacher Cotype]
    Let $X$ be a Banach space and let $q > 0$. The space $X$ is said to have Rademacher cotype $q$ if there exists a constant $C < \infty$ such that for all $n \in \mathbb{N}$ and $x_1, ..., x_n \in X$,
    \[
    \mathbb{E}_\sigma\norm{\sum\limits_{j = 1}^n \sigma_j x_j}^q \geq \frac{1}{C^q}\sum\limits_{j = 1}^n \norm{x_j}^q,
    \]
    where the expectation is taken over uniformly chosen $\sigma \in \{-1, 1\}^n$. The infimum over all constants $C > 0$ for which the above inequality holds is denoted $C_q(X)$.
    For the sake of duality, it is convenient to say every Banach space has cotype $\infty$.
\end{definition}
Rademacher type and cotype generalize the Pythagorean identity probabilistically. We note that any Hilbert space has both Rademacher type and cotype $2$ with constants $T = C = 1$. Indeed, since $\mathbb{E}[\sigma_i\sigma_j] = 0$ for $i \neq j$ and $\mathbb{E}[\sigma_i^2] = 1$ for all $1 \leq i, j \leq n$, we see
\[
\mathbb{E}_\sigma\norm{\sum\limits_{j = 1}^n \sigma_j x_j}^2 = \mathbb{E}_\sigma\left\langle \sum\limits_{j = 1}^n \sigma_j x_j, \sum\limits_{j = 1}^n \sigma_j x_j\right\rangle = \sum\limits_{i, j = 1}^n x_i x_j \mathbb{E}[\sigma_i\sigma_j] = \sum\limits_{i = 1}^n \norm{x_i}^2.
\]
In fact, Kwapie\'{n}'s theorem (Proposition 3.1 \cite{kwapien1972isomorphic}) states any Banach space with Rademacher type and cotype $2$ is a Hilbert space. 

Additionally, it can be shown that every Banach space has type $1$ (this is just the triangle inequality) and cannot have Rademacher type greater than 2 or Rademacher cotype less than 2 (Remark 6.2.11(a) ~\cite{albiac2016topics}). Thus, we say a Banach space has nontrivial type if it has Rademacher type greater than $1$ and nontrivial cotype if it has finite cotype. Note that if a Banach space has Rademacher type $p$ and cotype $q$, then it has Rademacher type $p'$ for $1 \leq p' \leq p$ and Rademacher cotype $q'$ for $q \leq q' \leq \infty$.

The following theorem establishes the Rademacher type and cotype for $L_p(\mu)$ spaces.
\begin{theorem}[Theorem 6.2.14~\cite{albiac2016topics}]
    Let $1 \leq p \leq \infty$. Then $L_p(\mu)$ has Rademacher type $\min\{p, 2\}$ and Rademacher cotype $\max\{2, p\}$.
\end{theorem}

Clearly, Rademacher type and cotype are \textit{linear} notions, as they rely on the addition and scalar multiplication of vectors. This presents an obstacle to directly applying them to the study of nonlinear aspects of Banach space theory. However, Rademacher type and cotype are \textit{local} notions, only relying on finitely many vectors at a time. Thus, as mentioned in the introduction, Ribe's result leads one to hope for a reformulation of these properties solely in terms of metric structure.

In \cite{enflo1978infinite}, Enflo formulated a metric notion related to work on  infinite-dimensional versions of Hilbert's fifth problem (although it was only proven recently that Enflo type coincides with Rademacher type for Banach spaces \cite{ivanisvili2020rademacher}). This was used to show that to show that if $0 < p < 1$ and $1 < q < \infty$, then $L_p(\mu)$ is not uniformly homeomorphic to $L_q(\mu)$ (Corollary of Theorem 7, \cite{enflo1978infinite}).

Mendel and Naor provided the metric formulation of cotype and proved its equivalence to Rademacher cotype for Banach spaces in \cite{mendel2008metric}. Our results rely primarily on metric cotype, which we now define.
\begin{samepage}
\begin{definition}[Metric Cotype]
    Let $(M, d)$ be a metric space and let $q > 0$. The space $(M, d)$ is said to have metric cotype $q$ with constant $\Gamma$ if for any $n \in \mathbb{N}$, there exists $m \in 2\mathbb{N}$ such that for all $f:\mathbb{Z}_m^n \to M$,
    \begin{equation}
    \label{eq:metric_cotype}
    \tag{$\ast$}
    \sum\limits_{j = 1}^n \mathbb{E}_x\left[d\left(f\left(x + \frac{m}{2}e_j\right), f(x)\right)^q\right] \leq \Gamma^q m^q \mathbb{E}_{\epsilon, x}[d(f(x + \epsilon), f(x))^q],
    \end{equation}
    where the expectations are taken with respect to uniformly chosen $x \in \mathbb{Z}_m^n$, $\epsilon \in \{-1, 0, 1\}^n$, and $\{e_j\}_{j = 1}^n$ represent the standard unit vectors.\\\\
    For $n \in \mathbb{N}$ and $m \in 2\mathbb{N}$, we denote by $\Gamma_q(M; n, m)$ the infimum over all $\Gamma > 0$ for which \cref*{eq:metric_cotype} holds for all $f:\mathbb{Z}_m^n \to M$.\\\\
    Additionally, for $n \in \mathbb{N}$ and $\Gamma > 0$, we denote by $m_q(M; n, \Gamma)$ the smallest even natural number such that \cref*{eq:metric_cotype} holds for all $f:\mathbb{Z}_m^n \to M$.\\\\
    Finally, as before, for the sake of duality, it is convenient to say every metric space has metric cotype $\infty$.
\end{definition}
\end{samepage}
Many of the results of Mendel and Naor~\cite{mendel2008metric} rely on Fourier analysis on the discrete tori $\mathbb{Z}_m^n$. The following definitions all contain important notions which will be used to establish bounds on the coarse embeddability ratio between $L_p(\mu)$ spaces.
\begin{definition}[Walsh Functions, Rademacher Projections, and $K$-convexity]
    Let $X$ be a Banach space and let $n \in \mathbb{N}$. For $A \subseteq \{1, ..., n\}$, the Walsh function $W_A: \{-1, 1\}^n \to \{-1,1\}$ is given by 
    \[
    W_A(\sigma_1, ..., \sigma_n) = \prod\limits_{j \in A} \sigma_j.
    \]
    Using discrete Fourier analysis, we see every $f: \{-1, 1\}^n \to X$ can be written
    \[
    f(\sigma_1, ..., \sigma_n) = \sum\limits_{A \subseteq \{1, ..., n\}} \widehat{f}(A) W_A(\sigma_1, ..., \sigma_n),
    \]
    where
    \[
    \widehat{f}(A) = \mathbb{E}_\sigma\big[f(\sigma)W_A(\sigma)\big],
    \]
    and the expectation is taken with respect to uniformly chosen $\sigma \in \{-1, 1\}^n$. Now, for a given function $f: \{-1, 1\}^n \to X$, the Rademacher projection of $f$ is given by
    \[
    \mathbf{Rad}(f) = \sum\limits_{j = 1}^n \widehat{f}(\{j\}) W_{\{j\}}.
    \]
    For $p \geq 1$, the $K_p$-convexity constant of $X$, denoted $K_p(X)$, is the smallest constant $K$ such that for every $n \in \mathbb{N}$ and every $f:\{-1, 1\}^n \to X$,
    \[
    \mathbb{E}_\sigma\norm{\mathbf{Rad}(f)(\sigma)}^p \leq K^p\mathbb{E}_\sigma\norm{f(\sigma)}^p,
    \]
    where expectations are taken with respect to uniformly chosen $\sigma \in \{-1, 1\}^n$. We say $X$ is $K$-convex if $K_2(X) < \infty$.
\end{definition}
\begin{remark}
    Due to Kahane's inequality (see remarks preceding Theorem 2.1 in \cite{mendel2008metric}), if $p > 1$ and $X$ is $K$-convex, then $K_p(X) < \infty$.
\end{remark}
When establishing upper bounds on the coarse embeddability ratio between $L_p(\mu)$ spaces, we will also consider the case $0 < p < 1$.
\begin{definition}[$L_p$ spaces, $0 < p < 1$]
    Let $(X, \Sigma, \mu)$ be a measure space and let $0 < p < 1$. For a Borel measurable function $f:X \to \mathbb{R}$, we define
    \[
    \norm{f}_p = \left(\integrate{X}{}{\modulus{f}^p}{d\mu}\right)^{1/p}.
    \]
    From this, we define $L_p(X, \Sigma, \mu)$ to be the quasi-normed space of all Borel measurable functions $f:X \to \mathbb{R}$ such that $\norm{f}_p < \infty$, modulo equivalence $\mu$-almost everywhere, equipped with the quasi-norm $\norm{\,\cdot\,}_p$.
\end{definition}
The above definition is the same as the definitions traditionally given for $1 \leq p < \infty$. The key difference here is that $\norm{\,\cdot\,}_p$ is no longer a norm for $0 < p < 1$ and instead a quasi-norm. Furthermore, for $0 < p < 1$, we have that for all $f, g \in L_p(X, \Sigma, \mu)$
\[
\norm{f + g}_p^p \leq \norm{f}_p^p + \norm{g}_p^p.
\]
Thus, by setting $d(f, g) = \norm{f - g}_p^p$ for all $f, g \in L_p(X, \Sigma, \mu)$, we obtain a metric on $L_p(X, \Sigma, \mu)$, which induces a metric topology. This metric topology is taken as the standard topology on $L_p(X, \Sigma, \mu)$ for $0 < p < 1$.

For $0 < p < 1$, the space $L_p(X, \Sigma, \mu)$ is complete with the quasi-norm $\norm{\,\cdot\,}_p$, but the lack of convexity makes analysis of $L_p(\mu)$ spaces much more difficult compared to the case $1 \leq p \leq \infty$.

For $0 < p < 1$, since $L_p(\mu)$ with $\norm{\,\cdot\,}_p$ is not a metric space, we cannot directly analyze them using metric cotype, but bounds obtained using metric cotype will still be used by considering $\norm{\,\cdot\,}_p^{p/2}$ (in fact, we will use this technique for $0 < p < 2$, since metric cotype has a lower bound at $2$).

The following three results are essential in establishing our upper bound on the coarse embeddability ratio between $L_p(\mu)$ spaces. They all concern specific quantitative bounds on constants appearing in the definition of metric cotype.
\begin{lemma}[Lemma 2.3~\cite{mendel2008metric}]
    \label{lemma5.6}
    Let $(M, d)$ be a metric space with at least two points. Then for all $n \in \mathbb{N}$, $\Gamma > 0$, and $q > 0$,
    \[
    m_q(M; n, \Gamma) \geq \frac{n^{1/q}}{\Gamma}.
    \]
\end{lemma}
\begin{theorem}[Theorem 4.1~\cite{mendel2008metric}]
    \label{thm5.7}
    Let $X$ be a $K$-convex Banach space with cotype $q$. Then for all $n \in \mathbb{N}$ and $m \in 4\mathbb{N}$,
    \[
    m \geq \frac{2n^{1/q}}{C_q(X)K_q(X)} \Longrightarrow \Gamma_q(X; n, m) \leq 15C_q(X) K_q(X)
    \]
\end{theorem}
\begin{lemma}[Lemma 7.1~\cite{mendel2008metric}]
    \label{lemma5.8}
    Let $(M, d)$ be a metric space, $n \in \mathbb{N}$, $\Gamma > 0$, and $0 <  q \leq p$. Then for all $\phi:\ell_p^n \to M$ and $s > 0$,
    \[
    n^{1/q}\kappa_\phi(2s) \leq \Gamma m_q(M; n, \Gamma) \omega_\phi\left(\frac{2\pi sn ^{1/p}}{m_q(M; n, \Gamma)}\right).
    \]
\end{lemma}
The following lemma provides a way for the choice of $m$ in Theorem \ref{thm5.7} to be arbitrarily close to its lower bound. This will be used to obtain our final bound on the coarse embeddability ratios from $L_p$ into $L_q$, $p > q \geq 2$.
\begin{samepage}
\begin{lemma}
    \label{lemma5.9}
    Let $M > 0$ and $q > 1$. For $n \in \mathbb{N}$, define
    \[
    A(n) = \frac{2n^{1/q}}{M}.
    \]
    For all $\epsilon > 0$ and $C > 0$, there exists $n \in \mathbb{N}$ with $n > C$ such that there exists $m \in 4\mathbb{N}$ with $A(n) \leq m < A(n) + \epsilon$. 
\end{lemma}
\begin{proof}
    Define $A:(0, \infty) \to (0, \infty)$ by
    \[
    A(x) = \frac{2x^{1/q}}{M}.
    \]
    Note that $A'$ is strictly decreasing since $q > 1$. By the Mean Value Theorem, for all $x \in (0, \infty)$, we see there exists $c_x \in (x, x + 1)$ such that 
    \[
    A'(c_x) = A(x + 1) - A(x).
    \]
    Thus,
    \[
    A'(x) > A'(c_x) = A(x + 1) - A(x).
    \]
    Now, let $\epsilon > 0$. Let
    \[
    R = \left(\frac{\epsilon M q}{2}\right)^{\frac{q}{1 - q}}.
    \]
    Suppose that $x > R$. Then
    \[
    A'(x) < A'(R) = \frac{2R^{\frac{1 - q}{q}}}{Mq} = \epsilon.
    \]
    Thus, for all $x > R$, we see $A(x + 1) - A(x) < \epsilon$, that is, $A(x + 1) < A(x) + \epsilon$. Now, since $\lim\limits_{x \to \infty} A(x) = \infty$, we see by the Intermediate Value Theorem and the fact that $A$ is increasing that there exists $m \in 4\mathbb{N}$ such that $A(x) = m$ and $x - 1 > \max\{R, C\}$. Let $n$ be the unique natural number in $(x - 1, x]$. Then $A(n) \leq A(x) = m \leq A(n + 1) < A(n) + \epsilon$.
\end{proof}
\end{samepage}
\begin{proposition}
    \label{prop5.10}
    Let $p > q \geq 2$. Then $\mathcal{CR}(L_p, L_q) \leq 30$\footnote{We remark GPT-5.5 Pro was able to extend this result to $\mathcal{CR}(L_p, L_q) \leq 4$ using an inequality the author was not aware of. The proof is somewhat similar, relying on metric cotype inequalities for $q$-barycentric targets presented in \cite{eskenazis2019nonpositive}.}
\end{proposition}
\begin{proof}
    First, note that $\ell_p$ embeds isometrically into $L_p$, so it suffices to show $\mathcal{CR}(\ell_p, L_q) \leq 30$. Let $\phi: \ell_p \to L_q$ and suppose $\mathcal{R}(\phi) > 30$. Applying Lemma \ref{lemma2.4}, let $\Delta > 0$ and $\delta > 30$ such that for all $x, y \in \ell_p$,
    \[
    \norm{x - y}_p \geq \Delta \Rightarrow \norm{\phi(x) - \phi(y)}_q \geq \delta
    \]
    and
    \[
    \norm{x - y}_p \leq 1 \Rightarrow \norm{\phi(x) - \phi(y)}_q \leq 1.
    \]
    Set $M = C_q(L_q)K_q(L_q)$ and let $\epsilon > 0$. Since $L_q$ has Rademacher (and metric) cotype $q$, we see $C_q(L_q)$ is finite. Since $L_q$ has nontrivial type, we see by Theorem 2.1 of ~\cite{pisier1982holomorphic} that $L_q$ is $K$-convex, so $K_q(L_q) < \infty$. Let $N \in \mathbb{N}$ such that $n^{1/p - 1/q} < \frac{1}{15 M \Delta \pi}$ for all $n \geq N$ (we may achieve this since $1/p - 1/q < 0$) and choose $n \in \mathbb{N}$ using Lemma \ref{lemma5.9} with $C = N$. Applying Lemma \ref{lemma5.8} with $s = \Delta/2$, we see for all $\Gamma > 0$,
    \begin{equation}
    \label{eq:prop5.10_lemma5.8_usage}
    \tag{$\diamondsuit$}
    n^{1/q}\kappa_\phi(\Delta) \leq \Gamma m_q(L_q; n, \Gamma) \omega_\phi\left(\frac{\Delta \pi n^{1/p}}{m_q(L_q; n, \Gamma)}\right).
    \end{equation}
    By choice of $n$, we see there exists $m \in 4\mathbb{N}$ with
    \begin{equation}
    \label{eq:prop5.10_m_bound}
    \tag{$\clubsuit$}
    \frac{2n^{1/q}}{M} \leq m < \frac{2n^{1/q}}{M} + \epsilon.
    \end{equation}
    Now, taking $\Gamma = \Gamma_q(L_q; n, m)$, we see by Theorem \ref{thm5.7} that $\Gamma \leq 15M$.
    Observe that
    \[
    m_q(L_q; n, \Gamma) \leq m < \frac{2n^{1/q}}{M} + \epsilon.
    \]
    By Lemma \ref{lemma5.6} and the fact $\Gamma \leq 15M$, we see by choice of $n$ that
    \begin{equation}
    \label{eq:prop5.10_omega_bound}
    \tag{$\spadesuit$}
    \omega_\phi\left(\frac{\Delta \pi n^{1/p}}{m_q(L_q; n, \Gamma)}\right) \leq \omega_\phi\left(\Gamma \Delta \pi n^{1/p - 1/q}\right) \leq \omega_\phi\left(15M \Delta \pi n^{1/p - 1/q}\right) \leq \omega_\phi(1) \leq 1.
    \end{equation}
    Thus, since $\Gamma \leq 15M$, we see by combining \Cref*{eq:prop5.10_lemma5.8_usage,eq:prop5.10_m_bound,eq:prop5.10_omega_bound} that
    \[
    \delta \leq \kappa_\phi(\Delta) \leq \frac{\Gamma m_q(L_q; n, \Gamma)}{n^{1/q}} \omega_\phi\left(\frac{\Delta \pi n^{1/p}}{m_q(L_q; n, \Gamma)}\right) \leq \frac{15M}{n^{1/q}} \left(\frac{2n^{1/q}}{M} + \epsilon\right) \leq 30 + \frac{15M\epsilon}{n^{1/q}}.
    \]
    Taking $\epsilon$ to be arbitrarily small (we also will obtain $n \to \infty$ by the dependence of $n$ on $\epsilon$), we see that $\delta \leq 30$, a contradiction. Thus, $\mathcal{CR}(\ell_p, L_q) \leq 30$, so $\mathcal{CR}(L_p, L_q) \leq 30$.
\end{proof}
\begin{corollary}
    \label{cor5.12}
    Let $p > 2$ and $0 < q < 2$. Then $\mathcal{CR}(L_p, L_q) \leq 30^{2/q}$.
\end{corollary}
\begin{proof}
    By Theorem 7 of ~\cite{bretagnolle1966lois} \footnote{Theorem 7 of ~\cite{bretagnolle1966lois} only states the result for $1 \leq q \leq 2$. In the proof of Theorem 7.3 of ~\cite{mendel2008metric}, they state the result holds for $0 < q \leq 2$. The proof in the case $0 < q < 1$ should be similar. One can follow the proof of Theorem 7 of ~\cite{bretagnolle1966lois} exactly, until the final step, where one must show the ultraproduct of $L_q(\mu)$ is itself an $L_q(\mu)$ space. This statement is true due to Proposition 3.3 of ~\cite{schreiber1972quelques}.}, there exists an isometric embedding $T$ of $L_q$ with the metric $\norm{x - y}_q^{q/2}$ into $L_2$. Suppose $\phi: L_p \to L_q$ with $\mathcal{R}(\phi) > 30^{2/q}$. By Lemma \ref{lemma2.4}, there exist $\Delta, \delta, \Lambda, \lambda > 0$ such that for all $x, y \in L_p$
    \[
    \norm{x - y}_p \geq \Delta \Rightarrow \norm{\phi(x) - \phi(y)}_q \geq \delta,
    \]
    \[
    \norm{x - y}_p \leq \Lambda \Rightarrow \norm{\phi(x) - \phi(y)}_q \leq \lambda,
    \]
    and $\frac{\delta}{\lambda} > 30^{2/q}$. Thus, we see for all $x, y \in L_p$,
    \[
    \norm{x - y}_p \geq \Delta \Rightarrow \norm{T\phi(x) - T\phi(y)}_2 \geq \delta^{q/2},
    \]
    \[
    \norm{x - y}_p \leq \Lambda \Rightarrow \norm{T\phi(x) - T\phi(y)}_2 \leq \lambda^{q/2}.
    \]
    By Lemma \ref{lemma2.4} in the reverse direction, 
    \[
    \mathcal{R}(T \circ \phi) \geq \frac{\delta^{q/2}}{\lambda^{q/2}} > 30.
    \]
    This shows $\mathcal{CR}(L_p, L_2) > 30$, contradicting Proposition \ref{prop5.10}. Thus, $\mathcal{CR}(L_p, L_q) \leq 30^{2/q}$.
\end{proof}
\begin{remark}
    From the previous two results, we see if $p > 2$ and $p > q \geq 1$, then $\mathcal{CR}(L_p, L_q) \leq 900$. For all other cases (that is, $p \leq 2$ or $p \leq q$), it is known (see \cite{bretagnolle1966lois, mendel2004euclidean, wojtaszczyk1996banach}) that $L_p$ coarsely embeds into $L_q$, so $\mathcal{CR}(L_p, L_q) = \infty$.
\end{remark}

%% file: 06_resolution.tex
\section{Spaces of Arbitrarily Large Finite Coarse Embeddability Ratio}
\label{sec:resolution}
In \cite{rosendal2023separationratiosmapsbanach}, Rosendal posed the following problem.
\begin{problem}
    Let $X$ and $E$ be Banach spaces with $\mathcal{CR}(X, E) > 1$. Does it follow that $\mathcal{CR}(X, E) = \infty$?
\end{problem}
In this section, we resolve this problem in the negative, finding Banach spaces $X_k, E_k$ such that $k \leq \mathcal{CR}(X_k, E_k) < \infty$. First, we introduce a graph induced by a separable metric space and a countable partition.
\begin{definition}[Graph induced by a partition]
    Let $(M, d)$ be a separable metric space and let $\{A_i\}_{i = 1}^\infty$ be a countable partition of $M$. The graph $\mathcal{G}(M, \{A_i\}_{i = 1}^\infty)$ is the graph on $\mathbb{N}$ by saying two distinct natural numbers $n, k$ are connected if $d(A_n, A_k) \leq 1$.
\end{definition}
The following definition allows us to define a norm using a graph on $\mathbb{N}$. We first define the norm on $c_{00}$, then extend it to an equivalent norm on $\ell_1$ using the lemma following the definition.
\begin{definition}[Graph-induced norm on $c_{00}$]
    Let $G$ be a graph on $\mathbb{N}$ and let $k \geq 1$. We define
    \[
    \mathcal{L}_k(G) = \{f: \mathbb{N} \to [0, k] : \modulus{f(i) - f(j)} \leq 1 \text{ whenever $i$ is connected to $j$ in $G$}\}
    \]
    For $a = \sum\limits_{i = 1}^n a_i e_i \in c_{00}$, we define
    \[
    \norm{a}_{G, k} = \max\left\{\frac{1}{2}\sum\limits_{i = 1}^n \modulus{a_i}, \sup\limits_{f \in \mathcal{L}_k(G)}\modulus{\sum\limits_{i = 1}^n a_i f(i)}\right\}.
    \]
\end{definition}
\begin{lemma}
    Let $G$ be a graph on the natural numbers and let $k \geq 1$. For $a = \sum\limits_{i = 1}^n a_i e_i \in c_{00}$,
    \[
    \frac{1}{2}\norm{a}_1 \leq \norm{a}_{G, k} \leq k\norm{a}_1.
    \]
\end{lemma}
\begin{proof}
    Clearly,
    \[
    \norm{a}_{G, k} = \max\left\{\frac{1}{2}\norm{a}_1, \sup\limits_{f \in \mathcal{L}_k(G)}\modulus{\sum\limits_{i = 1}^n a_i f(i)}\right\} \geq \frac{1}{2}\norm{a}_1.
    \]
    Let $f \in \mathcal{L}_k(G)$. For all $i \in \mathbb{N}$, we see by definition that $\modulus{f(i)} \leq k$. Thus, by H\"{o}lder's inequality, we see
    \[
    \modulus{\sum\limits_{i = 1}^n a_i f(i)} \leq \sum\limits_{i = 1}^n \modulus{a_i}\modulus{f(i)} \leq k\sum\limits_{i = 1}^n \modulus{a_i} = k \norm{a}_1.
    \]
\end{proof}
\begin{definition}[Graph-induced norm on $\ell_1$]
    Let $G$ be a graph on $\mathbb{N}$ and let $k \geq 1$. By the previous lemma, the norm $\norm{\,\cdot\,}_{G, k}$ on $c_{00}$ extends to an equivalent norm on all of $\ell_1$. We define $E_k(G) = (\ell_1, \norm{\,\cdot\,}_{G, k})$ to be the resulting space.
\end{definition}
The following lemma allows us to concretely compute the graph-induced norm in the case of differences in the standard unit vectors. This will allow us to apply Lemma \ref{lemma2.4} to estimate the coarse embeddability ratio from an $L_p$ space into $E_k$.
\begin{lemma}
    \label{lemma6.5}
    Let $G$ be a graph on $\mathbb{N}$ and let $k \geq 1$. Let $d_G$ denote the graph metric on $G$ (with $d_G(i, j) = \infty$ if $i$ and $j$ lie in different connected components). For all $i, j \in \mathbb{N}$,
    \[
    \norm{e_i - e_j}_{G, k} = \min\{d_G(i, j), k\}.
    \]
\end{lemma}
\begin{proof}
    If $i = j$, the claim is trivially true. Suppose $i \neq j$. Then 
    \[
    \frac{1}{2}\norm{e_i - e_j}_{1} = 1 \leq \min\{d_G(i, j), k\}.
    \]
    Additionally, if $f \in \mathcal{L}_k(G)$, then $\modulus{f(i) - f(j)} \leq d_G(i, j)$ by definition of $\mathcal{L}_k(G)$ and $\modulus{f(i) - f(j)} \leq k$ since $f(\mathbb{N}) \subseteq [0, k]$. Thus, we see
    \[
    \norm{e_i - e_j}_{G, k} \leq \min\{d_G(i, j), k\}.
    \]
    For the reverse inequality, fix $i$ and define $f_i(n) = \min\{d_G(i, n), k\}$ for each $n \in \mathbb{N}$. Clearly, $f_i \in \mathcal{L}_k(G)$. Furthermore, $f_i(j) = \min\{d_G(i, j), k\}$ and $f_i(i) = 0$, so
    \[
    \norm{e_i - e_j}_{G, k} \geq \min\{d_G(i, j), k\}.
    \]
    Thus,
    \[
    \norm{e_i - e_j}_{G, k} = \min\{d_G(i, j), k\}.
    \]
\end{proof}
Our final theorem allows us to resolve the stated problem of Rosendal.
\begin{theorem}
    \label{thm6.6}
    Let $2 < p < \infty$ and let $(X, \Sigma, \mu)$ be a separable measure space. Let $Z = L_p(X, \Sigma, \mu)$. For each $k \in \mathbb{N}$, there exists a renorming $E_k$ of $\ell_1$ such that
    \[
    k \leq \mathcal{CR}(Z, E_k) < \infty.
    \]
\end{theorem}
\begin{proof}
    Let $\{A_i\}_{i = 1}^\infty$ be a partition of $Z$ with $\operatorname{diam}(A_i) < 1$ for all $i \in \mathbb{N}$. Define $\mathcal{G} = \mathcal{G}(Z, \{A_i\}_{i = 1}^\infty)$. Set $E_k = (\ell_1, \norm{\,\cdot\,}_{\mathcal{G}, k})$. Define $\phi: Z \to E_k$ for all $x \in Z$ by $\phi(x) = e_i$ when $x \in A_i$. Let $x, y \in Z$ and let $i, j \in \mathbb{N}$ such that $\phi(x) = e_i$ and $\phi(y) = e_j$. Suppose $\norm{x - y} \leq 1$. Then, $d(A_i, A_j) \leq 1$, so $d_\mathcal{G}(i, j) \leq 1$. Thus, we see
    \[
    \norm{\phi(x) - \phi(y)}_{\mathcal{G}, k} = \min\{d_\mathcal{G}(i, j), k\} \leq \min\{1, k\} = 1.
    \]
    Now, suppose $\norm{x - y} \geq 2k + 2$. Let $x' \in A_i$, $y' \in A_j$. Then $\norm{x - x'} \leq 1$ and $\norm{y - y'} \leq 1$. Thus, we see 
    \[
    \norm{x' - y'} \geq \norm{x - y} - \norm{x - x'} - \norm{y - y'} \geq 2k + 2 - 2 = 2k.
    \]
    Taking the infimum over all choices of $x'$ and $y'$, we see $d(A_i, A_j) \geq 2k$. We claim this implies $d_\mathcal{G}(i, j) > k$. Indeed, if $d_\mathcal{G}(i, j) \leq k$, then there is a path $i = i_0, i_1 ..., i_m = j$ with $m \leq k$ and $d_\mathcal{G}(i_{r - 1}, i_r) = 1$ for all $1 \leq r \leq m$. However, for each $1 \leq r \leq m$, this means $d(A_{i_{r - 1}}, A_{i_r}) \leq 1$. Since $\operatorname{diam}(A_{i_r}) \leq 1$ for all $1 \leq r \leq m - 1$, we see
    \[
    d(A_i, A_j) \leq \sum\limits_{r = 1}^m d(A_{i_{r - 1}}, A_{i_r}) + \sum\limits_{r = 1}^{m - 1} \operatorname{diam}(A_{i_r}) \leq 2m - 1 < 2k,
    \]
    a contradiction. Thus, $d_\mathcal{G}(i, j) \geq k$, so
    \[
    \norm{\phi(x) - \phi(y)}_{\mathcal{G}, k} = \min\{d_\mathcal{G}(i, j), k\} = k.
    \]
    Therefore, by Lemma \ref{lemma2.4}, we see $\mathcal{R}(\phi) \geq k$. Thus, $\mathcal{CR}(Z, E_k) \geq k$. Now, note that $E_k$ is isomorphic to $\ell_1$. Thus, if $\mathcal{CR}(Z, E_k) = \infty$, then $\mathcal{CR}(Z, \ell_1) = \infty$. However, by Corollary \ref{cor5.12}, we see $\mathcal{CR}(Z, \ell_1) < \infty$, so $\mathcal{CR}(Z, E_k) < \infty$.
\end{proof}

%% file: 99_declaration.tex
\section*{Declaration of AI Use}
GPT-5.5 Pro was used to generate an initial proof in the case of $Z = \ell_p$, $p > 2$ for Theorem \ref{thm6.6}.
The author rewrote the argument provided by GPT-5.5 Pro into the current form, and extended it to the case $Z = L_p(X, \Sigma, \mu)$, $p > 2$ for a separable measure space $(X, \Sigma, \mu)$.
All other sections (Sections 1 - 5) were developed and written by the author without the use of LLMs.

%% file: refs.bib
@article{rosendal2023separationratiosmapsbanach,
  title={Separation ratios of maps between Banach spaces},
  author={Rosendal, Christian},
  journal={Comptes Rendus. Math{\'e}matique},
  volume={361},
  number={G10},
  pages={1663--1672},
  year={2023}
}

@article{kalton2007coarse,
  title={Coarse and uniform embeddings into reflexive spaces},
  author={Kalton, Nigel J},
  journal={The Quarterly Journal of Mathematics},
  volume={58},
  number={3},
  pages={393--414},
  year={2007},
  publisher={Oxford University Press}
}

@article{Baudier_2018,
   title={The coarse geometry of Tsirelson’s space and applications},
   volume={31},
   ISSN={1088-6834},
   url={http://dx.doi.org/10.1090/jams/899},
   DOI={10.1090/jams/899},
   number={3},
   journal={Journal of the American Mathematical Society},
   publisher={American Mathematical Society (AMS)},
   author={Baudier, Florent and Lancien, Gilles and Schlumprecht, Thomas},
   year={2018},
   month=feb, pages={699–717}
}

@inproceedings{James1972SOMESP,
  title={Some self-dual properties of normed linear spaces},
  author={James, Robert C},
  booktitle={Symposium on infinite dimensional topology},
  pages={159--176},
  year={2016},
  organization={Princeton University Press}
}

@article{mendel2008metric,
  title={Metric cotype},
  author={Mendel, Manor and Naor, Assaf},
  journal={Annals of Mathematics},
  pages={247--298},
  year={2008},
  publisher={JSTOR}
}

@article{mercourakis2013equilateralsetsinfinitedimensional,
 ISSN = {00029939, 10886826},
 URL = {http://www.jstor.org/stable/23807562},
 author = {S. K. Mercourakis and G. Vassiliadis},
 journal = {Proceedings of the American Mathematical Society},
 number = {1},
 pages = {205--212},
 publisher = {American Mathematical Society},
 title = {Equilateral Sets in Infinite Dimensional Banach Spaces},
 urldate = {2026-04-08},
 volume = {142},
 year = {2014}
}

@inproceedings{bretagnolle1966lois,
  title={Lois stables et espaces {$L^p$}},
  author={Bretagnolle, Jean and Dacunha-Castelle, Didier and Krivine, Jean-Louis},
  booktitle={Annales de l'IHP Probabilit{\'e}s et statistiques},
  volume={2},
  pages={231--259},
  year={1966}
}

@article{albiac2016topics,
  title={Topics in Banach Space Theory},
  author={Albiac, Fernando and Kalton, Nigel J},
  journal={Graduate Texts in Mathematics},
  year={2016},
  publisher={Springer},
  edition={2}
}

@article{tsirel1974not,
  title={Not every Banach space contains an imbedding of {$l_p$} or {$c_0$}},
  author={Tsirel'son, Boris S},
  journal={Functional Analysis and Its Applications},
  volume={8},
  number={2},
  pages={138--141},
  year={1974},
  publisher={Springer}
}

@article{figiel1974uniformly,
  title={A uniformly convex Banach space which contains no {$l_p$}},
  author={Figiel, Tadeusz and Johnson, William B},
  journal={Compositio Mathematica},
  volume={29},
  number={2},
  pages={179--190},
  year={1974}
}

@article{kwapien1972isomorphic,
  title={Isomorphic characterizations of inner product spaces by orthogonal series with vector valued coefficients},
  author={Kwapie{\'n}, Stanis{\l}aw},
  journal={Studia mathematica},
  volume={44},
  number={6},
  pages={583--595},
  year={1972}
}

@article{enflo1978infinite,
  title={On infinite-dimensional topological groups},
  author={Enflo, Per},
  journal={S{\'e}minaire Maurey-Schwartz},
  pages={1--11},
  year={1978}
}

@article{ivanisvili2020rademacher,
  title={Rademacher type and Enflo type coincide},
  author={Ivanisvili, Paata and Van Handel, Ramon and Volberg, Alexander},
  journal={Annals of mathematics},
  volume={192},
  number={2},
  pages={665--678},
  year={2020},
  publisher={Department of Mathematics, Princeton University Princeton, New Jersey, USA}
}

@article{pisier1982holomorphic,
  title={Holomorphic semi-groups and the geometry of Banach spaces},
  author={Pisier, Gilles},
  journal={Annals of Mathematics},
  volume={115},
  number={2},
  pages={375--392},
  year={1982},
  publisher={JSTOR}
}

@article{mendel2004euclidean,
  title={Euclidean quotients of finite metric spaces},
  author={Mendel, Manor and Naor, Assaf},
  journal={Advances in Mathematics},
  volume={189},
  number={2},
  pages={451--494},
  year={2004},
  publisher={Elsevier}
}

@book{wojtaszczyk1996banach,
  title={Banach spaces for analysts},
  author={Wojtaszczyk, Przemyslaw},
  number={25},
  year={1996},
  publisher={Cambridge University Press}
}

@article{mazur1932transformationes,
  title={Sur les transformationes isom{\'e}triques d’espaces vectoriels norm{\'e}s},
  author={Mazur, Stanis{\l}aw},
  journal={CR Acad. Sci. Paris},
  volume={194},
  pages={946},
  year={1932}
}

@article{kadets1967proof,
  title={Proof of the topological equivalence of all separable infinite-dimensional Banach spaces},
  author={Kadets, Mikhail Iosifovich},
  journal={Functional Analysis and Its applications},
  volume={1},
  number={1},
  pages={53--62},
  year={1967},
  publisher={Kluwer Academic Publishers-Plenum Publishers New York}
}

@article{torunczyk1981characterizing,
  title={Characterizing Hilbert space topology},
  author={Toru{\'n}czyk, Henryk},
  journal={Fundamenta Mathematicae},
  volume={111},
  number={3},
  pages={247--262},
  year={1981}
}

@article{ribe1978uniformly,
  title={On uniformly homeomorphic normed spaces II},
  author={Ribe, Martin},
  journal={Arkiv f{\"o}r Matematik},
  volume={16},
  number={1},
  pages={1--9},
  year={1978},
  publisher={Kluwer Academic Publishers Dordrecht}
}

@book{gromov2007metric,
  title={Metric structures for Riemannian and non-Riemannian spaces},
  author={Gromov, Mikhail},
  year={2007},
  publisher={Springer}
}

@article{kalton2012uniform,
  title={The uniform structure of Banach spaces},
  author={Kalton, Nigel J},
  journal={Mathematische Annalen},
  volume={354},
  number={4},
  pages={1247--1288},
  year={2012},
  publisher={Springer}
}

@inproceedings{rosendal2017equivariant,
  title={Equivariant geometry of Banach spaces and topological groups},
  author={Rosendal, Christian},
  booktitle={Forum of Mathematics, Sigma},
  volume={5},
  pages={e22},
  year={2017},
  organization={Cambridge University Press}
}

@article{randrianarivony2006characterization,
  title={Characterization of quasi-Banach spaces which coarsely embed into a Hilbert space},
  author={Randrianarivony, N.},
  journal={Proceedings of the American Mathematical Society},
  volume={134},
  number={5},
  pages={1315--1317},
  year={2006}
}

@article{naor2015uniform,
  title={Uniform nonextendability from nets},
  author={Naor, Assaf},
  journal={Comptes Rendus. Math{\'e}matique},
  volume={353},
  number={11},
  pages={991--994},
  year={2015}
}

@inproceedings{schreiber1972quelques,
  title={Quelques remarques sur les caract{\'e}risations des espaces {$L^p, 0 \leq p < 1$}},
  author={Schreiber, Michel},
  booktitle={Annales de l'institut Henri Poincar{\'e}. Section B. Calcul des probabilit{\'e}s et statistiques},
  volume={8},
  number={1},
  pages={83--92},
  year={1972}
}

@article{eskenazis2019nonpositive,
  title={Nonpositive curvature is not coarsely universal},
  author={Eskenazis, Alexandros and Mendel, Manor and Naor, Assaf},
  journal={Inventiones mathematicae},
  volume={217},
  number={3},
  pages={833--886},
  year={2019},
  publisher={Springer}
}
